\documentclass[10pt, reqno]{amsart}

\usepackage{amsthm,amsfonts,amsmath,amssymb,mathrsfs,enumitem,comment,mathtools,euler}
\usepackage[mathscr]{eucal}
\usepackage[utf8]{inputenc}
\usepackage[T1]{fontenc}
\usepackage[colorlinks=true, linktoc=page, linkcolor=webgreen, citecolor=webgreen, urlcolor=webbrown]{hyperref}
\usepackage{datetime}
\usepackage[initials, alphabetic]{amsrefs}
\usepackage[dvipsnames]{xcolor}

\usepackage{geometry}
\usepackage{setspace}
\definecolor{webgreen}{rgb}{0,.4,0}
\definecolor{webbrown}{rgb}{.4,0,0}

\renewcommand{\dateseparator}{, }
\newcommand{\todaymy}{\shortmonthname ~ {\the\day}\dateseparator \the\year}

\makeatletter
\def\blfootnote{\xdef\@thefnmark{}\@footnotetext}
\makeatother

\newtheorem{Thm}{Theorem}[section]

\newtheorem{Prop}[Thm]{Proposition}
\newtheorem{Cor}[Thm]{Corollary}

\theoremstyle{definition}
\newtheorem*{ack}{Acknowledgements}
\theoremstyle{remark}
\newtheorem{Def}[Thm]{Definition}

\newtheorem{Rem}[Thm]{Remark}
\newtheorem{convention}[Thm]{Convention}
\newtheorem{Exp}[Thm]{Example}
\numberwithin{equation}{section}

\newcommand{\<}{\langle}
\renewcommand{\>}{\rangle}

\DeclareMathOperator{\Real}{Re}
\DeclareMathOperator{\Imag}{Im}

\DeclareMathOperator{\id}{\mathbf{1}}
\DeclareMathOperator{\del}{\mathbf{d}}
\DeclareMathOperator{\E}{\mathbf{E}}
\DeclareMathOperator{\Sc}{\mathbf{S}}

\def\H{\mathscr H}
\def\L{\mathscr L}

\def\D{\mathscr D}

\def\Del{\boldsymbol{\Delta}}
\def\Di{\mathbf{D}}
\def\Re{\mathbb R}
\def\Rat{\mathbb Q}

\def\Co{\mathbb C}

\def\Z{\mathbb Z}
\def\N{\mathbb N}

\def\e{\mathbb E}
\def\B{\mathbf{B}}
\def\A{\mathbf{A}}

\def\pr{\mathbb{P}}

\def\ba{\mathbf{a}}
\def\ts{\tilde{\mathbf{s}}}
\def\bs{\mathbf{s}}

\def\eps{\varepsilon}

\begin{document}
\title[]{Probabilistic renormalization and analytic continuation}
\author[]{Gunduz Caginalp}
\author[]{Bogdan Ion}
\address{Department of Mathematics, University of Pittsburgh, Pittsburgh, PA 15260}
\email{caginalp@pitt.edu}
\email{bion@pitt.edu}
\date{}
\blfootnote{Aug 29, 2020}

\begin{abstract}
We introduce a theory of probabilistic renormalization for series, the renormalized values being encoded in the expectation of a certain random variable on the set of natural numbers. We identify a large class of weakly renormalizable series of Dirichlet type, whose analysis depends on the properties of a (infinite order) difference operator that we call Bernoulli operator. For the series in this class, we show that the probabilistic renormalization is compatible with analytic continuation. The general zeta series for $s\neq 1$ is found to be strongly renormalizable and its renormalized value is given by the Riemann zeta function. 
\end{abstract}

\maketitle

\thispagestyle{empty}

\section{Introduction}

The term \emph{renormalization} essentially refers to a collection of methodologies with a common approach rather than a specific methodology. In physical terms, the idea is that when one has a very large number of interacting particles that have a cooperative behavior such as a divergence at a critical temperature (e.g., a statistical mechanics system consisting of spins on a two-dimensional or three-dimensional lattice), then one can reduce the size of the system by averaging within a particular geometric configuration. For example, one can consider squares of neighboring particles, and average the spins in each square in some way. Without rescaling their interaction strength, however, one can expect only a trivial result such as $0$ or $\infty.$ Thus, the second, crucial step in renormalization is to rescale the interaction strength in such a way as to obtain a finite fixed point that is unique. By developing this methodology, Wilson \cite{WK} showed that key divergence exponents that could previously be calculated only with a \textit{tour de force} of mathematical methods, could be in fact calculated in a very simple manner (e.g., by solving a quartic equation). Renormalization can also be considered in a broader context, as shown in the book by Creswick, Farach and Poole \cite{CFP}, who demonstate the methology on Cantor sets, fractal shapes and random walks before discussing statistical mechanics models. In particular they show that classical results such as the Central Limit Theorem arise as a consequence of this approach of averaging and rescaling.

The analysis of divergent series and integrals is often crucial for our current understanding in certain branches of physics, and the objects that appear in the physical contexts usually also have number-theoretical relevance (e.g. the analysis of the Casimir effect \cite{RV}, or the two regularizations of the $H=xp$ model \cites{BK,Con}; see also \cite{SH} for a general review on physical approaches to the Riemann Hypothesis). The main techniques through which the particular cases that appear in physical contexts are typically handled are regularization and analytic continuation, but the literature contains a larger range of techniques that approach the phenomena of divergence from many points of view.  The standard reference in this respect is Hardy's  \emph{Divergent Series} \cite{Har} (see also \cite{Ber}) which describes an entire range of summability methods that
can be used to handle divergent series by starting with new definitions that render the series convergent in a classical sense.

To our knowledge, a systematic approach of divergent series and integrals based on renormalization ideas has not yet been developed. We initiate here a theory of \emph{probabilistic renormalization} for (divergent) series and we expect that a corresponding framework can be developed for integrals as well as for higher dimensional series and integrals. The basic idea, stemming from the original renormalization techniques, is that the \qq{divergent} part of the series can be canceled out by subtracting some  (possibly re-scaled) version of the same series, thus leaving behind some finite value that can be interpreted as the difference between the \qq{convergent} parts of the two series. There are, however, two subtleties that we emphasize below.

Let $\ba=(a_n)_{n\geq 0}$ denote a sequence and fix $\pr$ a probability measure on $\N$ for which any subset of $\N$ is measurable. Let $m\in \N$. We consider the sequence of differences of partial series for $\sum_{n\geq 0} a_n$ and  $\sum_{n\geq 0} a_{mn}$ as follows. For $0\leq j< m$ consider
$\bs_{[j]}$ defined by
\begin{equation}\label{ps1}
\bs_{[j],n}=\sum_{i=0}^{n} a_i -\sum_{i=0}^{(n-j)/m} a_{mi},
\end{equation}
and seen as a function of $n$ that is defined \emph{only} for $n$ such that $(n-j)/m$ is an integer (hence $n$ and $j$ are in the same residue class modulo $m$). The first subtlety has to do with the extension of $\bs_{[j]}$ to a function on $\N$. For this, we consider the operators $\E_m$  and $\Del_m=\E_m-\id$. The operator $\E_m$ acts on the function $\bs_{[j]}$ as $\E_m \bs_{[j],n}=s_{[j], n+m}$, for any $n\equiv j \mod m$, and $\id$ is the identity operator. For $h>0$, we denote
 \begin{equation}\label{ps2}
 \E^h_{m}=\sum_{n=0}^\infty \binom{h}{n}\Del_m^n.
 \end{equation}
 The definition and notation are motivated by the fact that the expression \eqref{ps2} is the Newton expansion for $(\id+\Del_m)^h$, and so, it serves as replacement for the action of the fractional shift operator $\E_m^h$.
The extension of  $\bs_{[j]}$ to a function on $\N$ is then defined as
$$
s_{[j], n}=\E_{m}^{\{\frac{n-j}{m}\}}  \bs_{[j], m\lfloor\frac{n-j}{m}\rfloor+j},\quad \text{for } n\geq j,
$$
and by assigning some arbitrary value for $n<j$; $\{N\}$ and $\lfloor N\rfloor$ denote the fractional part and, respectively, the integer part of $N$. 

Such an extension exists if the series \eqref{ps2} converges when applied to the function \eqref{ps1}. The extended functions $\bs_{[0]},\dots,\bs_{[m-1]}$ are regarded as random variables on $\N$, and their role in our setting is that the difference between the partial sums  for the series $\sum_{n\geq 0} a_n$ and  $\sum_{n\geq 0} a_{mn}$ takes the values specified by each $\bs_{[j]}$ with probability $1/m$. The random variable 
\begin{equation}\label{ps3}
X_{\ba}=\frac{1}{m}\left( \bs_{[0]}+\cdots+\bs_{[m-1]} \right),
\end{equation}
thus represents the function that gives the average value for the difference of the two partial sums. Its expectation $\e[X_\ba]$  with respect to $\pr$ can thus be regarded as the expected value of the difference between the two series. A successful cancelation of the \qq{divergence} occurs when $\e[X_\ba]<\infty$.  If this phenomenon is observed systematically for all possible scaling factors $m$ then we say that the series $\sum_{n\geq 0} a_n$ is \emph{weakly renormalizable}.   If there exists $s,S\in \Co$ such that 
\begin{equation}
\e[X_\ba]=S(1-m^s)\quad \text{for all } m\geq 1,
\end{equation}
then we say that the series is \emph{strongly renormalizable} and its renormalized value is $S$. We expect that a strongly renormalizable series would necessarily exhibit some \emph{scale invariance} property.

The second subtlety is the use of a \emph{finitely-additive}, translation invariant $\pr$. While one can use usual countably-additive measures of $\N$ to compute the expectation, we argue that the use of \emph{finitely-additive} measures has certain advantages. First,  we regard the translation invariance property as a substitute for the non-existent uniform probability measure on $\N$. Note that there are no countably-additive probability measures that are translation invariant. Thus, the translation invariance guarantees that the measure does not act as a regularization factor and thus $\e[X_\ba]$ accurately reflects the actual values of  $X_\ba$.  Second, a finitely-additive, translation invariant assigns zero measure to any finite set and so the values of $X_\ba$ on any finite set of $\N$ (such as the arbitrary values of $\bs_{[j]}$ assigned for $n<j$) do not alter $\e[X_\ba]$. 

This set-up is inspired by the first author's observation \cite{Cag}*{\S3} that renormalization ideas (scaling and averaging) can be used to assign a finite (renormalized) value for the series $\sum_{n\geq 1} n$. The set-up 
there is using scaling differently and the use of probability is only partially justified. For example, for the scaling factor $m=2$, consider the partial sum
$U_{n}=\sum_{k=1}^{n}k$ and $Y_{n}=U_{n}-4U_{\left\lfloor \frac{n}%
{2}\right\rfloor }$ so that
\[
Y_{n}=\left\{
\begin{array}
[c]{ccc}%
-n/2 & \text{if} & n \text{ even}\\
\left(  n+1\right)  /2 & \text{if} & n  \text{ odd}%
\end{array}
\right.  .
\]
Heuristically, one can consider choosing $n\in\N$ at random. If the
probability of $n$ being odd/even is $1/2$, then the expectation of  $Y_{n}$ is 
$$
\e\left[  U_{n}-4U_{\left\lfloor \frac{n}{2}\right\rfloor }\right]   
=\frac{1}{2}\left(  -\frac{n}{2}\right)  +\frac{1}{2}\left(  \frac{n+1}%
{2}\right)  =\frac{1}{4}=(1-4)\left(-\frac{1}{12}\right).
$$
The value $-1/12=\zeta(-1)$ is then regarded as the renormalized value of the
divergent series  $\sum_{n\geq 1} n$.

Our main result, Theorem \ref{t2}, identifies a large class of weakly renormalizable series. These series depend on the choice of a function $f(s,t) : \Co\times \Re_+\to \Co$ with special properties (holomorphic in $s$, differentiable in $t$, such that for a certain range of $s$, $f(s,t)$ is in the image of the Laplace-Mellin transform). We refer to Definition \ref{def-H} for the precise details. For a fixed $s\in \Co$ and $t_0>0$ we consider the series with terms
$a_n=f(s-1,t_0+n-1)$, that is the Dirichlet-type series
\begin{equation}
\Di^f(1-s,t_0)=\sum_{n=0}^\infty f(s-1,t_0+n).
\end{equation}
This series is convergent for $s$ in some right half-plane and  it has analytic continuation to $s\in\Co_+$ with a possible pole at $s=0$. We still denote this analytic continuation by $\Di^f(1-s,t_0)$. Our main result shows that, if $s\neq 0$, this series is weakly renormalizable and 
\begin{equation}
 \e[X_\ba]=\Di^f(1-s,t_0)-\Di^{\Sc_m f}(1-s,\frac{t_0+m-1}{m}),
 \end{equation}
 where $(\Sc_m f) (s,t)=f(s,tm)$. Note the occurrence on the right-hand side of the analytic continuation of the Dirichlet series. It perhaps desirable to assign to the series $\displaystyle\sum_{n= 0}^\infty a_n$ the renormalized value $\Di^f(1-s,t_0)$ given by analytic continuation and, consequently, to the series $\displaystyle\sum_{n= 0}^\infty a_{mn}$ the renormalized value $\Di^{\Sc_m f}(1-s,\frac{t_0+m-1}{m})$. These terms cannot be distinguished canonically from $\e[X_\ba]$ without further assumptions on  properties of $\e[X_\ba]$ as a function of $m$. Nevertheless, Theorem \ref{t2} shows that, at least for this class of examples, the probabilistic renormalization and the analytic continuation are compatible. This class of examples has the remarkable property that $X_\ba(n)$ is constant for large $n$ and hence the expectation $\e[X_\ba]$ is given by this constant value and does not depend on the choice of the finitely-additive measure $\pr$. The proofs make use of the properties of the Bernoulli operators, recently introduced and studied in \cite{Ion}. 
 
 Our second main result identifies our only example of strongly renormalizable series in this class, which arises for $f(s,t)=t^s$ and $t_0=1$. As expected, this function exhibits scale invariance properties
 $$f(s,mt)=m^sf(s,t).$$ We show that, if $s\neq 1$, the series 
 \begin{equation}
 \sum_{n\geq 1} \frac{1}{n^s}
 \end{equation}
is strongly renormalizable, and its renormalized value is $\zeta(s)$. Therefore, the probabilistic renormalization value coincides with the value produced by the analytic continuation.

We would also like to briefly comment on the applicability of our renormalization framework to the zeta functions associated to fractal strings (see, e.g., \cite{LF}), whose study was developed mainly through the efforts of M. Lapidus and collaborators. In particular, the zeta function of a \emph{self-similar} fractal string has, by construction, strong scale invariance properties that suggest that such series would constitute a large class of examples of strongly renormalizable series. The partial sums of these series (when divergent) have exponential growth, and the scale of the interaction in \eqref{ps1} (which, in this formulation, is used to compare re-scaled versions of the same series when the partial sums have at most polynomial growth)  has to be appropriately modified. We hope that an appropriate probabilistic renormalization framework for series with exponential growth and the necessary techniques to analyze them would be developed in the future.

\begin{ack}
We thank the referee for bringing to our attention the work of M. Lapidus, M. van Frankenhuijsen, and collaborators on the zeta functions associated to fractal strings and higher dimensional analogues.
The work of BI was partially supported by the Simons Foundation grant 420882. 
\end{ack}

\section{Notation}

\subsection{}\label{gamma}  We denote by $\Z$ and $\N$ the set integers and the set of positive integers, respectively. For $a\in \Z$, we denote by $\Z_{\geq a}$ and $\Z_{>a}$ the set of integers that are weakly larger and, respectively, strictly larger than $a$. Throughout, we reserve $u$ and $t$ to denote real variables with domain $\Re_{+}=(0,\infty)$; accordingly, $du$ and $dt$ refer to the Lebesgue measure on $\Re_{+}$. All integrals with respect to these measures are Lebesgue integrals of real or complex-valued measurable functions. All spaces of functions that will be considered, in particular the domains of all operators, are based on functions on $\Re_+$. Similarly, we reserve $z$ and $s$ to denote complex variables with domain $\Co$, unless otherwise specified. We will make use of the Gamma function $\Gamma(s)$ and the falling and raising factorials $s^{\underline{n}}=\Gamma(s+1)/\Gamma(s-n+1)$, $n\in \Z$.

\subsection{}
For $s_0\in \Co$ we denote left open half space, the right open half-space, and the right closed half-space determined by the line $\Real(s)=\Real(s_0)$ by
 $$
 \Co_{s_0^-}=\{s\in \Co~|~\Real(s)<\Real(s_0)\}, \quad 
 \Co_{s_0^+}=\{s\in \Co~|~\Real(s)>\Real(s_0)\}, \quad \text{and}\quad 
 \Co_{s_0}=\overline{\Co_{s_0^+}}.
 $$


\subsection{} We consider the series 
$$
\Di^f(s,t)=\sum_{n=0}^\infty {f(-s,t+n)},
$$
associated to  a function $f(s,t)$ that is holomorphic in $s\in \Co$. We will use the notation $\Di^f(s)$ for $$\Di^f(s,1)=\sum_{n=1}^\infty {f(-s,n)}.$$


\subsection{} To facilitate the verification of certain identities, we will adopt the following notation. For $f$ a function on $\Re_+$ and $t\in \Re_+$ we denote
$$
\<f,t\>=f(t).
$$

We use $\id$ to denote the identity operator, 
$\del$ to denote the derivative operator $\frac{d}{dt}$, $\E$ to denote the (forward) shift operator 
\[
\<\E f, t\>=\<f, t+1\>=f(t+1)
\]
and $\Del$ to denote the discrete (forward) derivative operator $\Del=\E-\id$,
\[
(\Del f)(t)=f(t+1)-f(t).
\]

\subsection{}
Moreover, for any fixed  $h>0$, we consider the corresponding operators:  $\Sc_h$ the scaling operator defined as
\[
\<\Sc_h f, t\>=\<f, ht\>,
\]
$\E_h$ the (forward) shift operator defined by 
\[
\<\E_h f, t\>=\<f, t+h\>,
\] and the difference  operator $\Del_h=\E_h-\id$. 
Note that, for any $h>0$, we have 
\[\
\E_h=\Sc_{h^{-1}}\E \Sc_{h},\quad \text{and} \quad \Del_h=\Sc_{h^{-1}}\Del \Sc_{h}.\] Indeed,
$$
\<\Sc_{h^{-1}}\E \Sc_{h} f, t\>=\<\E \Sc_{h} f, h^{-1}t\>=\<\Sc_{h} f, h^{-1}t+1\>=\<f, h(h^{-1}t+1)\>=\<f, t+h\>.
$$
In particular, any property of $\Del$ which is invariant under scaling is inherited by $\Del_h$.

\subsection{} \label{LT}

The Laplace-Lebesgue transform (or simply the Laplace transform) $\L(\varphi)$  of $\varphi$ is defined by 
\begin{equation}\label{eq2c}
\L(\varphi)(s)=\int_0^\infty e^{-su}\varphi(u)du
\end{equation}
for $s\in \Co$ for which the integral converges. As its domain we will consider $\D(\L)$, the $\Co$-vector space of functions $\varphi$ which are integrable on intervals $(0,R)$ for every $R>0$, and for which the integral \eqref{eq2c} is  \emph{absolutely} convergent for all $s\in \Co_{0^+}$. We denote by $\Imag(\L)$ the image of $\L$ on this domain. Note that $\varphi\in \D(\L)$ implies that $\varphi$ is locally integrable on $(0,\infty)$. Although holomorphic in $\Co_{0^+}$, we will mostly consider $\L(\varphi)$ as a function $\L(\varphi)(t)$ with $t\in\Re_+$. For any $k\geq 0$ its $k$-th derivative is given by 
\begin{equation}\label{eq2b}
\L(\varphi)^{(k)}(s)=\int_0^\infty e^{-su}(-u)^k\varphi(u)du.
\end{equation}
Note that if the integral \eqref{eq2c} converges absolutely for $s_0$ then it converges uniformly and absolutely in the closed half-plane 
 $ \Co_{s_0}$.


\subsection{} We define the Laplace-Mellin transform of $\varphi(u)\in \D(\L)$ as the function
$$
\L(\varphi(u)u^{s-1}/\Gamma(s))(t)
$$
as a function of two arguments $(s,t)\in \Co_{1^+}\times \Re_+$.  In other words, the Laplace-Mellin transform of $\varphi(u)$ for fixed $t$ is the Mellin transform of $e^{-tu}\varphi(u)/\Gamma(s)$. The Mellin transform of $\varphi(u)/
\Gamma(s)$ would correspond to evaluation of the Laplace-Mellin transform at $t=0$. For  $f=\L(\varphi)\in \Imag(\L)$ and $s\in \Co_{1^+}$ 
we denote 
$$
f_{-s} (t)=\L(\varphi(u)u^{s-1}/\Gamma(s))(t).
$$
In particular, one has $f_{-1}(t)=f(t)$. 
\section{Probabilistic renormalization}

\subsection{} We will need to consider a  probability measure $\pr$ on $\N$ (on the full power set $\sigma$-field). Some interesting examples of arithmetic origin, including the zeta distribution (also known as the Zipf distribution), can be found in in \cite{GolPro}. The expectation of a (bounded) random variable $X$ with respect to $\pr$ is denoted by $\e[X]$.

While the concept of renormalization that we consider here can be defined for any probability measure, in order for the renormalization process to produce a result that is as much as possible reflective to the properties of the original series, it is desirable to choose a measure which is to a certain extent uniform. For this reason, we will consider $\pr$ to be a \emph{finitely additive}, translation invariant,  probability measure.  For us, an important property of such a $\pr$ is that it assigns measure zero to any finite set. The theory of integration with respect to a finitely additive measure is carefully developed in 
\cite{DS}*{Chapter III, \S 1-3}.

On $\N$,  the existence of finitely additive measures $\pr$ for which any subset is measurable and which satisfy certain criteria of uniformity (for example, for any $m$, the measure of an equivalence class modulo $m$ is $1/m$, or the stronger property of being translation invariant) has been established in \cites{KO, SK} (see, for example, \cite{SK}*{Theorem 4.11}). The main results of this article are independent of the choice of $\pr$, but in general, the concepts of weakly and strongly renormalizable series that we define below do depend on such a choice.

\subsection{}  Let $\ba=(a_n)_{n\geq 0}$ denote a sequence. Consider the following operators acting on sequences: $\id$ the identity operator, $\E$ the forward shift operator, 
$(\E \ba)_n=a_{n+1}$, and $\Del=\E-\id$ the (forward) discrete derivative. For $h>0$, let 
 \[\E^h=\sum_{n=0}^\infty \binom{h}{n}\Del^n.\]
 The domain of the operator $\E^h$ consists of sequences $\ba$ for which the series defining $\E^h\ba$ is pointwise convergent. The expansion in the definition of the operator $\E_h$ is the Newton binomial expansion for $(\id+\Delta)^h$. As such, $\E^h$ is a version of the $h$-shift operator; we denote $$a_{n+h}=(\E^{h}\ba)_n.$$ For $h\in\Z_{\geq 0}$ the series becomes a finite sum and produces the known value of $a_{n+h}$.
 
\subsection{}\label{sec: ext} Fix $m\in \N$ and, for $j\in \Z_{\geq 0}$ denote by $[j]=\{j+mk~|~k\geq 0\}$ its equivalence class modulo $m$ (as an equivalence relation on $\Z_{\geq 0}$). Fix $\Lambda_m=\{0,\dots,m-1\}$ and $j\in \Lambda_m$. Let $\bs$ be a sequence defined only on $[j]$. More precisely, the sequence consists of the terms
$$
\bs_j, \bs_{j+m}, \bs_{j+2m}, \bs_{j+3m},\dots
$$
It might be convenient to think of a sequence as a function on $\Z_{\geq 0}$. In this light, the function $\bs$ is only defined on $[j]\subset \Z_{\geq 0}$.

We will consider the (potential) extension of the function $\bs$ to $\Z_{\geq j}$ as follows. Let $\ts$ denote the sequence $$\ts_n=s_{nm+j}, \quad n\geq 0.$$ If $\ts$ belongs to the domain of $\E^{i/m}$ for all $1<i< m$ then we denote 
$$
\tilde{s}_{n+i/m}=(\E^{i/m} \ts)_n, \quad \text{for all} \quad n\in \Z_{\geq 0},~ i\in \Lambda_m.
$$
We can now define $$s_n=\tilde{s}_{(n-j)/m}  \quad \text{for all} \quad n\in\Z_{\geq j}.$$
Under these circumstances we say that $\bs$ has an extension to $\Z_{\geq j}$. We use the same symbol, $\bs$, to denote the extension.
 
\subsection{} To write an expression for the extension $\bs$ without invoking $\ts$ we consider the operators $\E_m$  and $\Del_m=\E_m-\id$. The operator $\E_m$ acts naturally on the original sequence $\bs$ as $\E_m s_{n}=s_{n+m}$, for any $n\in[j]$. For $h>0$, we denote
 \[\E_{m}^{h}=\sum_{n=0}^\infty \binom{h}{n}\Del_m^n.\]
If $\bs$ has an extension to $\Z_{\geq j}$, the extension can then be described as
$$
s_n=\E_{m}^{\{\frac{n-j}{m}\}}\left(\bs\right)_{m\lfloor\frac{n-j}{m}\rfloor+j},
$$
where $\{N\}$ and $\lfloor N\rfloor$ denote the fractional part and, respectively, the integer part of $N$.

\subsection{} Let $m\in \N$ and $j\in \Lambda_m$. For the sequence $\ba$ consider the sequence $\bs_{[j]}$ defined on $[j]$ as follows
$$
\bs_{[j],n}=\sum_{i=0}^{n} a_i -\sum_{i=0}^{(n-j)/m} a_{mi}.
$$
As in \S\ref{sec: ext} we will consider the possible extension of $\bs_{[j]}$ to $\Z_{\geq j}$. If this extension exists, then
$$
\bs_{[j],n}=\E_{m}^{\{\frac{n-j}{m}\}}\left(\sum_{i=0}^{m\lfloor\frac{n-j}{m}\rfloor+j} a_i -\sum_{i=0}^{\lfloor \frac{n-j}{m}\rfloor} a_{mi}\right).
$$

The underlying intuition (present in all flavors of renormalization) behind the definition below is that the \qq{divergent} part of the series can be canceled out by subtracting some  (possibly re-scaled) version of the same series, thus leaving behind some finite value that can be interpreted as the difference between the \qq{convergent parts} of the two series. If this phenomenon is observed systematically for all possible scaling factors then we say that the series is weakly renormalizable. 

Let $(\N,\mathscr{P}(\N),\pr)$ be a probability space with the following properties
\begin{itemize}
\item $\mathscr{P}(\N)$ is the $\sigma$-field of all subsets of $\N$;
\item $\pr$ is a \emph{finitely-additive} probability measure;
\item $\pr(A+n)=\pr(A)$, for all $A\subseteq \N$ and $n\in \N$.
\end{itemize}
It is important to remark that there are no countably-additive probability measures that satisfy the third condition. While the definition below can be extended to include general (countably-additive) probability measures with no assumption of uniformity, the expectation of a random variable $X$ with respect to such a measure would not accurately reflect the values of $X$ and the procedure would be akin to the more widely used regularization procedures for series (with the measure $\pr$ playing the role of the regulator). 

The most basic random variables that will occur in our setting are as follows. Whenever  a sequence $\bs_{[j]}$ extends to $\Z_{\geq j}$, we will consider a further arbitrary extension of $\bs_{[j]}$ as a function on $\N$ and regard $\bs_{[j]}$ as a random variable on $\N$.

\begin{Def}\label{wr}
With the notation above, we say that the series corresponding to $\ba$ is weakly renormalizable (with respect to $\pr$) if, for all $m\in \N$, we have
\begin{itemize}
\item $\bs_{[j]}$ have extension to $\Z_{\geq j}$ for all $j\in \Lambda_m$;
\item $\e[X_\ba]$ is finite, where $X_{\ba}=\frac{1}{m}\left( \bs_{[0]}+\cdots+\bs_{[m-1]} \right)$.
\end{itemize}
\end{Def}
Note that there is a random variable $X_\ba$ for each $m$ but we do not include the reference to $m$ in the notation as it will be clear from the context. In section \ref{main} we will construct a large class of examples of weakly renormalizable sequences.

\begin{Rem}
The function $X_\ba$ is regarded as a random variable on $\Z_{\geq m}$. We can extend $X_\ba$ in any way to a function on $\N$. Because $\pr$ assigns measure zero to any finite set the expectation of  $X_\ba$ does not depend on the extension.
\end{Rem}

\begin{Rem} 
For any fixed $m$, the function $X_\ba$  can be considered as a moving average of differences between the partial sums corresponding to the sequence $(a_n)_{n\geq 0}$ and those corresponding to the sequence $(a_{mn})_{n\geq 0}$. However, it is important to note that we do not take the average of the $m$ numerical values of the partial sums
$$
\bs_{[n-m+1],n-m+1},~\bs_{[n-m+2],n-m+2},\dots,\bs_{[n],n}
$$
but rather the values at $n$ of the extensions of these partial sums, specifically on
$$
\bs_{[n-m+1],n},~\bs_{[n-m+2],n},\dots,\bs_{[n],n}.
$$
This subtlety is crucial because in practice the difference between the numerical values of the partial sums and those of their extension are enough to produce fluctuations in the values of $X_\ba$ that significantly alter its expectation.
\end{Rem}

\begin{Rem}
If the series associated to $\ba$ is weakly renormalizable then one would like to assign to $\sum_{n\geq 0} a_n$ and $\sum_{n\geq 0} a_{mn}$ some finite values (their renormalized values) $S$ and $S_m$ such that for any $m\in \N$ we have
$$
\e[X_\ba]=S-S_m.
$$
In general, there is no canonical way of assigning these values, and this is consistent with the general intuition that the \qq{divergent} part of a series can be thought of only up to a constant. If some canonical renormalized value can be assigned for one the series, then all the series acquire a canonical renormalized value. Thus, a divergent weakly renormalizable series is a series that is divergent not because of some loss of scale. If the partial terms of series have some \emph{scale invariance} properties, allowing us to argue that $S_m$ must equal some re-scaled version of $S$ then we can assign a renormalized value to $\sum_{n\geq 0} a_n$. This is the subject of Definition \ref{sr}.
\end{Rem}

\subsection{} Without pursuing this in full generality, let us point out that for \emph{convergent} series we expect that $\e[X_\ba]=S-S_m$, with $S$ and $S_m$ the actual values of $\sum_{n\geq 0} a_n$ and $\sum_{n\geq 0} a_{mn}$. We prove this under some technical assumption.

\begin{Prop}\label{prop-conv}
Let $\ba$ be a sequence such that the corresponding series is absolutely convergent and weakly renormalizable. If, for $m\in \N$ and any $j\in \Lambda_m$, $0<i<m$ the series $\E^{i/m}\bs_{[j]}$ is uniformly convergent, then 
$$
\e[X_\ba]=L-L_m,
$$
where $\displaystyle L=\sum_{n=0}^\infty a_n$ and $\displaystyle L_m=\sum_{n=0}^\infty a_{mn}$.
\end{Prop}
\begin{proof}
Let us remark first that if $X:\Z_{\geq 0}\to \Co$ has limit $A$ at infinity, then $\e[X]=A$. Indeed, by linearity, it is enough to show this for the case $A=0$. For any $\eps>0$, $|X(n)|<\eps$ for any $n$ outside a finite set. Since the measure of any finite set is $0$ we obtain that $|\e[X]|<\eps$. Therefore, $\e[X]=0$.

Our claim now follows from the fact that $X_\ba$ has limit $L-L_m$ at infinity, or rather, each extension $\bs_{[j]}$ has limit $L-L_m$ at infinity. This is definitely true about the partial sums $\bs_{[j]}$ and also about their extensions because the uniform convergence of $\E^{i/m}\bs_{[j]}$ implies that the limit at infinity of the series $\E^{i/m}\bs_{[j]}$ is the series of the limit. Since all sequences $\Del^n \bs_{[j]}$, $n\geq 1$, have limit $0$ at infinity, we obtain the desired result.
\end{proof}

\begin{Rem}
On the other hand, the failure to be renormalizable points to some loss of scale. For example, let us consider the case of the Grandi series, which is summable by other methods (e.g. Ces\`aro summation). Let, $\ba=(a_n)_{n\geq 0}$ be defined by $a_0=0$ and $a_n=(-1)^{n+1}$, $n\geq 1$. A straightforward computation gives the following:
$X_\ba(n)=1-1/m(n+1/2)$, if $m$ is even and $X_\ba(n)=1/2-1/(2m)$, if $m$ is odd. This shows that $\e[X_\ba]$ is finite when $m$ is odd, but not when $m$ is even. Therefore, the Grandi series is not weakly renormalizable. Ultimately, this happens because $\sum_{n\geq 0}a_n$ and $\sum_{n\geq 0}a_{mn}$ for even $m$ are divergent for different reasons. If our definition of renormalization would be based only on odd scaling factors $m$, we would find that 
$$
\e[X_\ba]=(1-m^{-1})\frac{1}{2}.
$$
Using the terminology of Definition \ref{sr} below, the series would be strongly renormalizable with renormalized value $1/2$ and critical exponent $-1$.
\end{Rem}

\subsection{} There is, however, a particular situation in which we can assign a canonical value to a weakly renormalizable series.  
\begin{Def}\label{sr}
Let $\ba$ be a sequence such that  the corresponding series is weakly renormalizable. We say that the series corresponding to $\ba$ is (strongly) renormalizable (with respect to $\pr$) if there exists $s,S\in \Co$ such that 
$$
\e[X_\ba]=S(1-m^s)\quad \text{for all } m\geq 1.
$$
The value $S$  will be denoted by $\Sigma a_n$ and will be called the renormalized value of $\displaystyle \sum_{n\geq 0}a_n$. The value $s$ will be called the critical (scaling) exponent of $\displaystyle \sum_{n\geq 0}a_n$.
\end{Def}

As already mentioned, and also in the light of Proposition \ref{prop-conv}, this is only expected if the sequence $\ba$ or the associated series exhibits some form of scale invariance. This is consistent with other flavors of renormalization, especially in the context of physical applications \cites{CFP, Coll}.

Our main results are concerned with a class of examples of weakly and strongly renormalizable series that depend analytically on one complex parameter and the relationship between the renormalized values and those arising from the analytical continuation of actual limits of the series from the domain of convergence to the maximal domain of holomorphy (in our case, the complex plane).


\section{The Bernoulli operator}  \label{BO}
\subsection{}\label{series-ad}

We will denote by  $\A$ the operator
$$
\A=\sum_{n=0}^\infty \E^{n}.
$$
This operator can be regarded as the discrete integral (i.e. summation) operator  associated to the counting measure on $\Z_{\geq 0}$.   We consider $\A$ as an operator with domain the vector space $\D(\A)$  of  $\Co$-valued functions $f(t)$ on $\Re_+$  for which the series
\begin{equation}\label{Aseries-def}
(\A f)(t)=\sum_{n=0}^\infty (\E^{n}f)(t)=\sum_{n=0}^\infty f(t+n)
\end{equation}
converges absolutely and locally uniformly in $t$.  In general, we do not expect that $\A$ is continuous with respect to any natural topology on its domain, but $\A$ will preserve the local integrability and continuity of the argument.
We use the corresponding definition and notation for $\A_h$, $h>0$, and its domain. More precisely,  $\A_h$ is the operator with domain the vector space $\D(\A_h)$  of  $\Co$-valued functions $f(t)$ on $\Re_+$  for which the series
\[
(\A_h f)(t)=\sum_{n=0}^\infty (\E_h^{n}f)(t)=\sum_{n=0}^\infty f(t+nh)
\]
converges absolutely and locally uniformly in $t$.

\subsection{}
To motivate the consideration of the operator $\A$ let us remark that $(\A f)(t)$ can be regarded as a discrete anti-derivative of $-f(t)$. Take, for example,  $f\in L^1(\Re_+)$ and denote 
\begin{equation}\label{eq1}
F(t)=(\A f)(t).
\end{equation}
This is a function that is a.e. finite and  in $L^1_{\rm \ell oc}(\Re_+)$. Indeed,
\[
\sum_{n=0}^\infty \int_0^1 |f(t+n)|dx=\int_0^\infty |f(t)|dt<\infty,
\]
and, by the Fubini-Tonelli theorem, 
\[
\int_0^1 |F(t)|dt \leq \int_0^1 \sum_{n=0}^\infty  |f(t+n)|dt<\infty.
\]
Furthermore, a.e. we have
\[
\Delta F(t)=F(t+1)-F(t)= -f(t).
\]
In other words, $F$ is a discrete anti-derivative of $-f$. Formally, we would like to write  $F=-\frac{1}{\Delta}f$, or $F=\frac{1}{\id-\E}f$, the latter being also consistent with  the usual formal expansion 
\[
\frac{1}{\id-\E}=\sum_{n=0}^\infty \E^n.
\]
The convergence of the series \eqref{eq1} depends only on the  behavior of the function in a neighborhood of $+\infty$, so the hypothesis that $f\in L^1(\Re_+)$ can be replaced by  $f\in L^1_\infty(\Re_+)$.

\subsection{} The following basic result is straightforward (see, e.g. \cite{Ion}*{Propposition 4.3}).

\begin{Prop}\label{p2}
Let $f=\L(\varphi)\in \Imag(\L)$ and $\varphi(u)=o(u^{\gamma+1})$, $u\to 0^+$, for some $\gamma\in \Co_{-1^+}$. Then $f\in \D(\A)$ and 
$$
(\A f)(t)=\L\left(\frac{\varphi(u)}{1-e^{-u}}\right).
$$
\end{Prop}


\subsection{} We denote by $\B$ the (infinite order) difference operator 
$$
\B=\frac{\ln(\id+\Del)}{\Del}=\sum_{n=0}^\infty \frac{(-\Del)^n}{n+1},
$$
which we call the Bernoulli operator associated to the standard geometric series $\displaystyle\sum_{n\geq 0} z^n$, or simply the Bernoulli operator. We refer to \cite{Ion}*{\S6} for details on the definition and properties of similar operators associated to other series. We consider $\B$ as an operator with domain the vector space $\D(\B)$  of  $\Co$-valued functions $f(t)$ on $\Re_+$  for which the series
\begin{equation}\label{series-def}
\sum_{n=0}^\infty \frac{(-\Del)^nf(t)}{n+1}
\end{equation}
converges absolutely and locally uniformly in $t$. In particular, on this domain, $\B$ preserves the local integrability and continuity of the argument. 

We use the corresponding definition and notation for $\B_h$, $h>0$, and its domain. Note that if $f(t)\in \D(\B_h)$ then $\E_h f(t)\in \D(\B_h)$ and, in particular, the series 
\[
\ln(\id+\Del_h)f(t):=-\sum_{n=0}^{\infty}\frac{(-\Del_h)^{n+1}}{n+1}f(t)
\]
converges locally uniformly.


\subsection{} For any fixed  $h>0$, we define
\[\E^h=\sum_{n=0}^\infty \binom{h}{n}\Del^n.\]
We consider $\E^h$ as an operator with domain the vector space $\D(\E^h)$  of  $\Co$-valued functions $f(t)$ on $\Re_+$  for which the series
\begin{equation}
\sum_{n=0}^\infty \binom{h}{n}\Del^nf(t)
\end{equation}
converges absolutely and locally uniformly in $t$. In particular, on this domain, $\E^h$ preserves the local integrability and continuity of the argument. 

The operator $\E^h$ is itself a Bernoulli operator, namely the operator associated with the function $$\alpha(z)=z^h.$$ The function $\alpha(z)$ has a branch point singularity at $z=0$ and is holomorphic in the complex plane with a branch cut along the negative real axis (if $h\not\in \Z_{\geq 0}$). It has the power series expansion 
\[
\sum_{n=0}^\infty \binom{h}{n}(z-1)^n
\]
around $z=1$ with radius of convergence $1$. The construction and properties of Bernoulli operators in \cite{Ion} have been developed for functions $\alpha(z)$ holomorphic in the unit disk with a multi-power series expansion around $z=1$. Nevertheless, all the results in \cite{Ion} hold verbatim for functions with a possible branch point singularity at $z=0$, in particular for $\alpha(z)=z^h$.

\subsection{} For the following results we refer to \cite{Ion}*{Proposition 6.21}.

\begin{Prop}\label{p1}
Let $f=\L(\varphi)\in \Imag(\L)$, $h>0$. Then, 
\begin{enumerate}[label={\roman*)}]
\item $f\in \D(\B_h)$;
\item $\displaystyle\B_h f=\L\left(\frac{hu}{1-e^{-hu}}\varphi(u)\right)$;
\item $\ln(\id+\Del_h)f=h\del f=\L(-hu\varphi(u))$.
\end{enumerate}
In particular, $\B$ can be considered as a linear operator
$
\B: \Imag(\L) \to \Imag(\L).
$
\end{Prop}

\begin{Prop}\label{p5}
Let $f=\L(\varphi)\in \Imag(\L)$, $h>0$. Then, 
\begin{enumerate}[label={\roman*)}]
\item $f\in \D(\E^h)$;
\item $\displaystyle\E^h f=\L\left(e^{-hu}\varphi(u)\right)=\E_h f$.
\end{enumerate}
In particular, $\E^h=\E_h$ on $\Imag(\L)$.
\end{Prop}

\begin{Cor} \label{cor1}
Let $f\in \Imag(\L)$, $h, h^\prime>0$, $m\in \N$. Then, 
\begin{enumerate}[label={\roman*)}]
\item $\B_{h^{-1}}\Sc_h f= \Sc_h\B f$;
\item $\B_h\E_{h^\prime}=\E_{h^\prime}\B_h$;
\item $\displaystyle\B f=\frac{1}{m}\sum_{i=0}^{m-1}\B_m \E_{i}f$.
\end{enumerate}
\end{Cor}
\begin{proof} Parts (i) and (ii) follow by direct verification using Proposition \ref{p1}~(ii). Part (iii) follows from Proposition \ref{p1}~(ii) and the identity
 $$\frac{u}{1-e^{-u}}=\frac{mu}{1-e^{-mu}}\cdot \frac{1}{m}\sum_{i=0}^{m-1}e^{-iu},\quad u\in \Re_+.$$
\end{proof}
\subsection{} The relationship between the operators $\A_h$ and $\B_h$ is the following.

\begin{Thm}\label{t1}
Let $f=\L(\varphi)\in \Imag(\L)$ and $\varphi(u)=o(u^{\gamma})$, $u\to 0^+$,  for some $\gamma\in\Co_{-1^+}$. Then
 $$\B_h f(t)=-\sum_{n=0}^\infty f^{\prime}(t+hn)=-\A_h f^{\prime}(t).$$
\end{Thm}


\subsection{} For the following result we refer to \cite{Ion}*{Proposition 7.1, Corollary 7.2}.
\begin{Prop}\label{p3}
Let $f=\L(\varphi)\in \Imag(\L)$, $h>0$ and $s\in \Co_{1^+}$. Then  
\begin{enumerate}[label={\roman*)}]
\item $f_{-s}(t)$ is holomorphic in $s\in \Co_{1^+}$ and 
$$
\del f_{-s}(t)=-s f_{-s-1}(t);
$$
\item
$f_{-s}\in \D(\B_h)$ and $\B_h f_{-s}(t)=(\B_h f)_{-s}(t)$;
\item If $\varphi(u)=o(u^\gamma)$, $u\to 0^+$, for some $\gamma\in\Co_{-1^+}$ then
$$
\B_h f_{-s}(t)=-\sum_{n=0}^\infty  f_{-s}^{\prime}(t+nh)=-\A_h f_{-s}^{\prime}(t);
$$
\item $\B_h f_{-s}(t)$ is holomorphic in $s\in \Co_{1^+}$.
\end{enumerate}
\end{Prop}

\subsection{} The following subspace of $\D(\L)$ consists of functions that are dominated by some increasing function in the domain of $\L$ 
\begin{equation}
\D^{\iota}(\L):=\{\varphi~|~ |\varphi|\leq \psi,~\text{for some increasing}~\psi\in \D(\L)\}.
\end{equation}
\begin{Exp}
Let $\varphi\in \D(\L)$ such that $\varphi$ is continuous, bounded in a neighborhood of $0$, and $\varphi(t)=o(t^\gamma)$, $t\to +\infty$, for some $\gamma>0$. Then, $\varphi\in \D^{\iota}(\L)$.
\end{Exp}
Following \cite{Ion}*{Definition 7.6} we consider the following space of functions.
\begin{Def}\label{def-H}
Let $\H$ denote the space of functions $f(s,t): \Co\times \Re_+\to \Co$ satisfying the following properties
\begin{itemize}
\item $f(s,t)$ is differentiable in $t\in \Re_+$;
\item $f(s,t)$ and $\del f(s,t)$ are holomorphic in $s\in \Co$;
\item $f(t):=f(-1,t)=\L(\varphi)$ for some $\varphi\in \D^{\iota}(\L)$;
\item $f(-s,t)=f_{-s}(t)$ for $(s,t)\in \Co_{1^+}\times \Re_+$.
\end{itemize}
\end{Def}
\begin{convention} When we discuss functions $f(s,t)\in \H$ we assume that $f(t)$ is the corresponding function in the context of  Definition \ref{def-H}. We also employ the notation $f_s(\cdot)=f(s,\cdot)$ for any $s\in \Co$.
\end{convention}
We note that the space considered in \cite{Ion}*{Definition 7.6} allows for functions $f(s,t)$ with isolated singularities with respect to $s\in\Co$.  We restrict here to entire functions for simplicity. In particular, the elements of $\H$ satisfy the following properties.
\begin{Rem}\label{rem-deriv} Let $f(s,t)\in \H$. From Proposition \ref{p3} and the fact that $f(s,t)$ and $\del f(s,t)$ are entire, we obtain
$$
\del f(s,t)=s f(s-1,t)\quad \text{for all}\ (s,t)\in \Co\times\Re_+.
$$
In particular, $\del^{n} f(s,t)=s^{\underline{n}} f(s-n,t)$, and therefore $f(n,t)$ is a polynomial in $t$ of degree at most $n$.
\end{Rem}



\subsection{} We will use the following result \cite{Ion}*{Theorem 7.14, Corollary 7.16}.

\begin{Thm}\label{t2m}
For any $h>0$ the following hold
\begin{enumerate}[label={\roman*)}]
\item $\H\subset \D(\B_h)$ and $\H\subset \D(\E^h)$;
\item $\B_h: \H \to \H$ and $\E^h: \H \to \H$.
\end{enumerate}
\end{Thm}

\begin{Cor}\label{ac} Let $f(s,t)\in \H$ and assume that  $f=\L(\varphi)$ with $\varphi(u)=o(u^\gamma)$, $u\to 0^+$, for some $\gamma\in \Co_{-1^+}$. Then, $\B f(-s,t)$ is the analytic continuation of 
$$
s\Di^f(s+1,t)=s\A f(-s-1,t)=\B f(-s, t), \quad s\in \Co_{1^+}.
$$

\end{Cor}

\begin{Cor}\label{cor3}
Let $f(s,t)\in \H$ and assume that  $f=\L(\varphi)$ with $\varphi(u)=o(u^\gamma)$, $u\to 0^+$, for some $\gamma\in \Co_{-1^+}$. Then, for any $N\in \Z_{\geq 0}$, $s\in \Co$, we have
$$
s\sum_{n=0}^{N-1} f(-s-1,t+n)= \B f(-s,t)-\B f(-s,t+N).
$$
\end{Cor}
\begin{proof}
The functions on both sides of the equality are holomorphic in $s$ and agree for $s\in \Co_{1^+}$. In consequence, they agree for $s\in \Co$.
\end{proof}

\begin{Cor}\label{cor2} 
Let $f(s,t)\in \H$, $h>0$, $m\in \N$. Then, for $(s,t)\in \Co\times\Re$ we have
\begin{enumerate}[label={\roman*)}]
\item $\Sc_{h^{-1}}\B\Sc_h f_s= \B_h f_s$;
\item $\displaystyle \B f_s=\frac{1}{m}\sum_{i=0}^{m-1}\B_m \E_{i}f_s$;
\item $\displaystyle\<\B f_s , t\>=\frac{m^s}{m}\sum_{i=0}^{m-1}\<\B ~m^{-s}\Sc_m f_s , (t+i)/m\>$.
\end{enumerate}
\end{Cor}
\begin{proof}
Note that the function $h^{-s}\Sc_hf(s,t)=h^{-s}f(s,ht)\in \H$. The function and its derivative are clearly entire and 
$$
h^{s}f(-s,ht)=g_{-s}(t),\quad (s,t)\in \Co_{1^+}\times \Re_+,
$$
for $g(t)=\L(\varphi(h^{-1}u))$, verifying the remaining conditions. From Corollary \ref{cor1} 
$$
h^{-s}\Sc_{h^{-1}}\B ~h^s\Sc_h f(-s,t)= \B_h f(-s,t),\quad (s,t)\in \Co_{1^+}\times \Re_+,
$$
and since the functions on the both sides of the equality are entire we obtain the first claim. The second and third claim follow in a similar fashion from Corollary \ref{cor1}. 
\end{proof}

\begin{Cor}\label{cor4}
Let $h, h^\prime>0$. Then, on $\H$ we have
\begin{enumerate}[label={\roman*)}]
\item $\E^h=\E_h$;
\item $\B_h\E_{h^\prime}=\E_{h^\prime}\B_h$.
\end{enumerate}
\end{Cor}
\begin{proof}
The conclusions directly follow from Proposition \ref{p5}, Corollary \ref{cor1} and the fact that $\H$ consists of entire functions.
\end{proof}

\section{Main results}\label{main}

\subsection{}\label{seq} For this section let us fix $f(s,t)\in \H$ such that  $f=\L(\varphi)$ with $\varphi(u)=o(u^\gamma)$, $u\to 0^+$, for some $\gamma\in \Co_{-1^+}$.  Denote $f^m(s,t)=m^{-s}S_{m}f(s,t)\in \H$. Note that $f^m(t)=\L(\varphi(m^{-1}u))$ and $\varphi(m^{-1}u)=o(u^\gamma)$, $u\to 0^+$.
Also fix $t_0>0$ and define the sequence $\ba$ by $$a_0=0, \quad a_n=f(s-1,t_0+n-1),\quad n\geq 1.$$
The series corresponding to $\ba$ is therefore
$$
\sum_{n=0}^\infty f(s-1,t_0+n)=\Di^f(1-s,t_0).
$$
As assured by Proposition \ref{p3} iii), this series is convergent for $s$ in some right half-plane and, by Corollary \ref{ac},  it has analytic continuation to $s\in\Co_+$ with a possible pole at $s=0$. Our main result shows that, if $s\neq 0$, this series is weakly renormalizable and the expectation of $\e[X_\ba]$ is expressed in terms of the analytic continuation of $\Di^f(-s+1,t_0)$.

\begin{Thm}\label{t2}
If $s\neq 0$, the series corresponding to $\ba$ is weakly renormalizable and 
$$
\e[X_\ba]=-\frac{1}{s}\left(\B f(s,t_0)-m^s \B f^m\left(s,\frac{t_0-1}{m}+1\right)\right).
$$
\end{Thm}
\begin{proof}
Let $j\in \Lambda_m$ and $n\in [j]$. From Corollary \ref{cor3} we have
$$
\sum_{i=0}^n a_i=\frac{1}{s}\left(\B f(s,t_0+n)- \B f(s,t_0)\right).
$$
Similarly, we have
$$
 \sum_{i=0}^{(n-j)/m} a_{mi}= \sum_{i=1}^{(n-j)/m} f(s-1, t_0-1+mi)=m^s\sum_{i=1}^{(n-j)/m} f^m(s-1, \frac{t_0-1}{m}+i)
 $$
 which, again by Corollary \ref{cor3} equals
$$
 \frac{m^s}{s}\left(\B f^m\left(s,\frac{t_0-1}{m}+\frac{n-j}{m}+1\right)- \B f^m\left(s,\frac{t_0-1}{m}+1\right)\right).
$$
Therefore, 
\begin{equation*}
s_{[j],n}=\frac{1}{s}\left(\B f(s,t_0+n)- m^s\B f^m\left(s,\frac{t_0-1}{m}+\frac{n-j}{m}+1\right)\right)-\frac{1}{s}\left(\B f(s,t_0)-m^s \B f^m\left(s,\frac{t_0-1}{m}+1\right)\right).
\end{equation*}
Corollary \ref{cor4} i) applied to $\B f$ yields, for any $h>0$,
$$
\E^h \B f(s,t_0+n)=\E_h \B f(s,t_0+n)=\B \E_h f(s,t_0+n)= \B f(s,t_0+n+h),
$$
and similarly, $\E^h \B f^m(s,t_0+n)= \B f^m(s,t_0+n+h)$. This shows that the above formula for $s_{[j],n}$ is, in fact, valid for all $n\in \Z_{\geq m}$.

Now, for $n\in \Z_{\geq m}$,
\begin{equation*}
\begin{aligned}
X_\ba(n)=\frac{1}{m}\left(s_{[0],n}+\cdots+s_{[m-1],n}  \right)= &\frac{1}{s}\left(B f(s,t_0+n)-\frac{m^s}{m}\sum_{j=0}^{m-1}\B f^m\left(s,\frac{t_0-1}{m}+\frac{n-j}{m}+1\right)\right)   \\ & -\frac{1}{s}\left(\B f(s,t_0)-m^s \B f^m\left(s,\frac{t_0-1}{m}+1\right)\right).
\end{aligned}
\end{equation*}
Corollary \ref{cor2} iii) gives the following equality
\begin{equation*}
\begin{aligned}
B f(s,t_0+n)&= \frac{m^s}{m}\sum_{i=0}^{m-1}\B f^m\left(s,\frac{t_0+n+i}{m}\right)\\
&=\frac{m^s}{m}\sum_{j=0}^{m-1}\B f^m\left(s,\frac{t_0-1}{m}+\frac{n+1+m-1-j}{m}\right)\\
&=\frac{m^s}{m}\sum_{j=0}^{m-1}\B f^m\left(s,\frac{t_0-1}{m}+\frac{n-j}{m}+1\right).
\end{aligned}
\end{equation*}
Consequently, we have for $n\in \Z_{\geq m}$ the identity
$$
X_\ba(n)= -\frac{1}{s}\left(\B f(s,t_0)-m^s \B f^m\left(s,\frac{t_0-1}{m}+1\right)\right).
$$
Since this is a constant function, our claim immediately follows. It is important to remark that, because $X_\ba$ is constant, its expectation is independent of the choice of $\pr$. 
\end{proof}
\begin{Rem}
The value of $\e[X_\ba]$ can be expressed in terms of the analytic continuation of the Dirichlet series associated with $f(s,t)$ as follows
$$
\e[X_\ba]=\Di^f(1-s,t_0)-\Di^{\Sc_m f}(1-s,\frac{t_0+m-1}{m}).
$$
\end{Rem}
\begin{Rem}
It is tempting to argue that we should assign to the series $\displaystyle\sum_{n= 0}^\infty a_n$ the renormalized value $\Di^f(1-s,t_0)$ and, consequently, to the series $\displaystyle\sum_{n= 0}^\infty a_{mn}$ the renormalized value $\Di^{\Sc_m f}(1-s,\frac{t_0+m-1}{m})$. However, we do not know how to canonically distinguish these terms from $\e[X_\ba]$ without further assumptions on  properties of $\e[X_\ba]$ as a function of $m$. A strong argument in favor of such an outcome for the renormalization is the fact that $\Di^f(1-s,t_0)$ is indeed the value of the series for the values of $s$ for which the series is convergent. Nevertheless, Theorem \ref{t2} shows that, at least for this class of examples, the probabilistic renormalization and the analytic continuation are compatible. 
\end{Rem}

\subsection{} We will now ready to identify some strongly renormalizable series. Let $\ba$ be defined as in \S\ref{seq} for the choice of $t_0=1$. More precisely,
$$a_0=0, \quad a_n=f(s-1,n),\quad n\geq 1.$$
\begin{Rem}
Strong renormalization is expected only in the presence of some scale invariance for the corresponding sequence. In our context, 
the property $f(s,\cdot)=f^m(s,\cdot)$ translates to some scaling invariance property of $f(s,\cdot)$, specifically
\begin{equation}\label{eq2}
f(s,t)=m^{-s}f(s,tm).
\end{equation}
\end{Rem}
\begin{Prop}\label{t3}
If $s\neq 0$ and $f(s,\cdot)=f^m(s,\cdot)$ for all $m\geq 1$,  the series corresponding to $\ba$ is strongly renormalizable and 
$$
\e[X_\ba]=-(1-m^s)\frac{1}{s}\B f(s,1).
$$
Therefore, 
$$
\Sigma a_n=-\frac{1}{s}\B f(s,1)=\Di^f(1-s).
$$
\end{Prop}
\begin{proof}
From Theorem \ref{t1} for $m\geq 1$, and since $t_0=1$, we obtain that,
$$
\e[X_\ba]=-\frac{1}{s}\left(\B f(s,1)-m^s \B f^m\left(s,1\right)\right).
$$
Our hypothesis $f=f^m$ further simplifies this expression to 
$$
\e[X_\ba]=-\frac{1}{s}\B f(s,1)(1-m^s)=\Di^f(1-s),
$$
which proves our claim.
\end{proof}

\begin{Thm}\label{t4}
Let $s\neq 1$. The series associated to the sequence $\ba$ defined by $a_0=0$, $a_n=1/n^s$, $n\geq 1$, is strongly renormalizable and
$$
\e[X_\ba]=(1-m^s)\zeta(s).
$$
Therefore,
$$
\Sigma 1/n^s=\zeta(s),
$$
where $\zeta(s)$ the Riemann zeta function.
\end{Thm}
\begin{proof}
The function $f(s,t)=t^s$ satisfies the hypotheses of Proposition \ref{t3}.
\end{proof}
We explain below that this is (up to a scalar multiple) the only strongly renormalizable series that satisfies the hypothesis of Proposition \ref{t3}.

\subsection{} Let us recall the following class of functions, originally defined by Kubert \cites{Kub, Lan} in the context of (algebraic) distributions, which are crucial for the definition of the concepts of measure and integral in number theory, notably for the construction of $p$-adic integration in Iwasawa theory. See \cites{KL1,KL2,KL3} for related work.

\begin{Def} Let $s\in \Co$. A function $F:(0,1)\to \Co$ is called $s$-Kubert function if it satisfies the identity
\begin{equation}
F(t)=m^{s-1}\sum_{i=0}^{m-1} F\left(\frac{t+i}{m}\right), \quad \text{for all }  t\in(0,1), ~m\in \N.
\end{equation}
\end{Def}
In the original number-theoretical  context, the Kubert functions have domain $\Rat/ \Z$; Kubert also introduced higher dimensional versions of these functions. Our definition follows Milnor \cite{Mil} who proved the following result.

Recall the definition of the Hurwitz zeta function $\zeta(s,t)=\Di^{t^s}(s,t)$,
and of the periodic zeta function $\ell(s,t)$, which is the meromorphic continuation of the series
$$
\sum_{n=0}^\infty \frac{e^{2\pi i nt}}{n^s}.
$$
For $t\not \in \Z$, $\ell(s,t)$ is entire in $s$; for $t\in \Z$, $\ell(s,t)=\zeta(s)$.

\begin{Thm} The complex vector space consisting continuous $s$-Kubert functions is two dimensional and consists of real analytic functions. Furthermore, this vectors space is spanned by the functions $-s\zeta(1-s,t)$ and $\ell(s,t)$.
\end{Thm}

The statement is a combination of the results in  \cite{Mil}*{Theorem 1} and those in  \cite{Mil}*{Appendix 1}.

\subsection{}
We are now ready to show that the function $t^s$ is essentially the only element of $\H$ that satisfies the hypothesis of Proposition \ref{t2}.

\begin{Prop}
If $s\neq 0$ and $f(s,\cdot)=f^m(s,\cdot)$ for all $m\geq 1$, then $f(s,t)$ is a scalar multiple of $t^s$.
\end{Prop}
\begin{proof}
Since $\Delta \B f(s,t)=\del f(s,t)$ for $f(s,t)\in \H$, it suffices to argue that if $f$ satisfies the  hypotheses of Proposition \ref{t2} then $\B f$ is a scalar multiple of $\B t^s$. 
If $f(s,\cdot)=f^m(s,\cdot)$ for all $m\geq 1$ then Corollary \ref{cor2} iii) implies that $\B f_s$ is an $s$-Kubert function. The function $-s\zeta(1-s,t)=\B t^s$ and $t^s$ satisfies the hypothesis of Proposition \ref{t2}. We only have to argue that $\ell(s,t)$, $t\in(0,1)$, is not the restriction of a function in $\B(\H)$.

Indeed, if there exists $f\in \H$ so that for a fixed $s\neq 0$ we have $\B f(s,t)=\ell_s(t)$, $t\in (0,1)$, then by using $\Delta \B f(s,t)=\del f(s,t)$ we obtain that 
$f(s,t+1)=\ell_s(t)+\del \ell_s(t)$, $t\in (0,1)$ and, inductively,
$$
f(s,t+n)=\sum_{k=0}^n \binom{n}{k}(2\pi i)^k \ell_{s-k}(t), \quad \text{for } t\in (0,1).
$$
The function $f(s,t)$ is continuous in $t$. The condition of continuity at positive integers translates into 
$$
\sum_{k=1}^n \binom{n}{k}(2\pi i)^k \zeta(s-k)=0, \quad \text{for all } n\geq 1.
$$
This implies that $\zeta(s-n)=0$ for all $n\geq 1$. This is in contradiction with the fact that the zeros of the zeta function are either negative even integers or contained in the critical strip.
\end{proof}

\section{Examples}

\subsection{} We conclude with some natural examples of weakly renormalizable series. A readily available source of elements of $\H$ is provided by \cite{Ion}*{Theorem 7.14}, specifically, by the application of general Bernoulli operators to known elements of $\H$. In particular, we obtain a large class of elements of $\H$ by applying Bernoulli operators to $t^s\in\H$. Many known Dirichlet series are of this type (see, e.g., \cite{Ion}*{\S 9}). 

We will list below some sequences $\ba$ defined as in \S\ref{seq} for $t_0=1$ (for simplicity). Recall our notation for the falling factorial in \S\ref{gamma}.
Each sequence is of the form $a_n=f(s-1,n)$, $n\geq 1$. The function $f(s,t)$ will always be a function given by the analytical continuation to $s\in \Co$ of an absolutely convergent series in the half-plane $s\in \Co_{0^-}$ of the form 
\begin{equation}\label{eq3}
f(s,t)=s^{\underline{\nu}}\sum_{i=0}^\infty c_{i+1}(t+i)^{s-\nu},\quad s<0,
\end{equation}
for some $\nu\in \Z_{\geq 0}$, and some sequence of coefficients $(c_{i+1})_{i\geq 0}$  for which the generating series $$c(z)=\sum_{i\geq 0}c_{i+1}z^i$$ represents a function that is holomorphic in the unit disk and has a pole of order $\nu$ at $z=1$. The analytic continuation of the series \eqref{eq3} is guaranteed, for example, by \cite{Ion}*{Theorem 7.14}. Therefore, the terms of the sequence $\ba$ arise from the analytic continuation of the tails of the series \eqref{eq3}. More precisely,
\begin{equation}\label{eq4}
a_n=(s-1)^{\underline{\nu}}\sum_{i=n}^\infty c_{i+1-n}i^{s-1-\nu},\quad s<0.
\end{equation}

\subsubsection{} For $c(z)=1/(1-z)$, we have $\nu=1$, $c_{i}=1$. The sequence $\ba$ consists essentially of values of the Hurwitz zeta function at integer $t$ values. For $c(z)=1/(1+z)$, we have $\nu=0$, $c_{i}=(-1)^{i-1}$. In this case $\ba$ consists essentially of values of the Dirichlet-Hurwitz eta function at integer $t$ values.

Similarly, if $\chi:\Z\to \Co$ is a Dirichlet character of modulus $k$,  let
$
c(z)=1/({1-z^k})\sum_{i=1}^{k}\chi(i)z^{i-1}.
$
In this case $\nu$ equals $1$ or $0$ depending on whether $\chi$ is a principal character or not. The sequence $\ba$ consists of tails of classical Dirichlet L-series.

\subsubsection{} Let $P\subset \Re^d$ be a rational, convex $d$-polytope. Let $c_{n}=|nP\cap \Z^d|$, $n\geq 1$. The  series 
$$
Ehr(z)=1+\sum_{n=1}^{\infty} c_n z^n
$$
 is called the Ehrhart series of $P$. With our notation, $ Ehr(z)=1+zc(z)$. The integer $\nu$ can be at most $d+1$, depending on $P$. In this case, the terms of  $\ba$ arise from the analytic continuation of  the series
\begin{equation}
a_n=(s-1)^{\underline{\nu}}\sum_{i=1}^\infty |iP\cap \Z^d| (i+n-1)^{s-1-\nu},\quad s<0.
\end{equation}

\subsubsection{} Let $a,b\in \Co$ and $x \in (-1,1)$. The generating function for Jacobi polynomials $P_n^{(a,b)}(x)$  is given by \cite{AAR}*{Theorem 6.4.2}
$$
c(z)=\sum_{n=0}^{\infty} P_n^{(a,b)}(x)z^n=2^{a+b}R^{-1}(R+1-z)^{-a}(R+1+z)^{-b}, \quad \text{where}\ R=R(x,z)=(z^2-2xz+1)^{1/2}.
$$
The expressions refer to the principal branch of the complex power function with the branch cut  along the negative real axis. We have $\nu=0$ and 
the terms of  $\ba$ arise from the analytic continuation of  the series
\begin{equation}
a_n= \sum_{i=1}^\infty P_i^{(a,b)}(x)(i+n-1)^{s-1},\quad s<0.
\end{equation}

\subsection{} A second source of examples is of arithmetic origin; we only record one here. Let $k$ be an algebraic number field of degree $N$ and consider the corresponding Dedekind zeta function $$\zeta_k(s)=\sum_{n=1}^\infty \frac{c_n}{n^s}.$$ 
The coefficient $c_n$ counts the number of ideals of norm $n$ in the  integer ring of $k$. In this case, $c(z)$ is still holomorphic in the unit disk, but has the unit circle as a cut. In this case, the terms of  $\ba$ arise from the analytic continuation of  the series
\begin{equation}
a_n= s\sum_{i=1}^\infty c_i(i+n-1)^{s-1},\quad s<0.
\end{equation}

While all the examples above lead to weakly renormalizable series, in the absence of some scale invariance, they are not expected to be strongly renormalizable.


\end{document}